\documentclass[11pt]{amsart}
\usepackage{amssymb,amsmath,amsthm,enumitem,tikz,graphicx}
\theoremstyle{plain}
\newtheorem{theorem}{Theorem}[section]

\newcommand{\CC}{{\mathbb C}}
\newcommand{\RR}{{\mathbb R}}
\newcommand{\ZZ}{{\mathbb Z}}

\begin{document}

\title{The norm of the  resolvent of the Volterra operator}
\date{25 July 2022}

\author{Thomas Ransford}
\address{D\'epartement de math\'ematiques et de statistique, Universit\'e Laval,
Qu\'ebec City (Qu\'ebec),  Canada G1V 0A6.}
\email{thomas.ransford@mat.ulaval.ca}
%\thanks{Research supported by grants from NSERC and the Canada Research Chairs program.}

\begin{abstract}
In this expository note, 
we compute the exact value of the norm of the resolvent of the Volterra operator.
\end{abstract}

%\keywords{Volterra operator, resolvent norm}

%\subjclass[2010]{47G10}

\maketitle

%%%%%%%%%%%%%%%%%%%%%%%%%%%%%%%%%%%%%%%%%%%%%%%%

\section{Introduction}\label{S:intro}

Let $V: L^2[0,1]\to L^2[0,1]$ be the  Volterra operator, defined by
\[
Vf(x):=\int_0^x f(t)\,dt.
\]
It is well known that $V$ is a compact  quasi-nilpotent operator, and that its adjoint is given by
\[
V^*f(x)=\int_x^1 f(t)\,dt.
\]
It is also well known (see \cite[Problem~188]{Ha82}) that 
\begin{equation}\label{E:Halmos}
\|V\|=2/\pi.
\end{equation}
A great deal of research has been devoted to computing or estimating $\|V^n\|$ for higher powers of~$V$.
The  best  estimates, to our knowledge, are established in \cite{BD09}, 
which also contains a summary of the history of the problem.

For more general polynomials $p$, the problem of computing $\|p(V)\|$ was addressed by Lyubich 
and Tsedenabayar in \cite{LT10}. They outlined a program for calculating the singular values of $p(V)$,
and carried out this program in detail for the case when $p$ has degree one. 

The singular values of an invertible operator permit us to determine not only its norm
(the largest singular value) but also the norm of its inverse (the reciprocal of the smallest singular value).
Thus the computation the singular values of $I+\nu V$ carried out by Lyubich 
and Tsedenabayar
leads to a formula for the norms of $(I+\nu V)^{-1}$ and hence of the resolvent operators $(V-\mu I)^{-1}$.
They did not state the formula explicitly, so we give it here.

%%%%%%%%%%%%%%%%%%%%%

\section{Statement of the main result}\label{S:main result}

\begin{theorem}\label{T:main}
Let $\mu=\alpha+i\beta\in\CC\setminus\{0\}$.
\begin{enumerate}[label=\rm({\roman*}),topsep=0pt, itemsep=0pt]

\item\label{I:i} If $\alpha>\beta^2$, then 
\[
\|(V-\mu I)^{-1}\|=\frac{1}{\sqrt{\alpha^2-\gamma^2}},
\]
where $\gamma$ is the unique solution in $(0,\alpha)$ to the equation
\[
\coth\Bigl(\frac{\gamma}{\beta^2+\gamma^2}\Bigr)=\frac{\alpha}{\gamma}.
\]

\item\label{I:ii} If  $\alpha=\beta^2$, then
\[
\|(V-\mu I)^{-1}\|=\frac{1}{\alpha}.
\]

\item\label{I:iii}  If  $\alpha<\beta^2$ and $\beta\ne0$, then
\[
\|(V-\mu I)^{-1}\|=\frac{1}{\sqrt{\alpha^2+\gamma^2}},
\]
where $\gamma$ is the unique positive number such that  $\gamma/(\beta^2-\gamma^2)\in(0,\pi)$ and
\[
\cot\Bigl(\frac{\gamma}{\beta^2-\gamma^2}\Bigr)=\frac{\alpha}{\gamma}.
\]

\item\label{I:iv} If  $\alpha< 0$ and $\beta=0$, then
\[
\|(V-\mu I)^{-1}\|=1/|\alpha|.
\]
\end{enumerate}
\end{theorem}

Figure~\ref{F:domain} illustrates the various curves and regions in $(\alpha,\beta)$-space in which the parts~\ref{I:i}--\ref{I:iv}
of the theorem apply. The origin is excluded.

\begin{figure}[ht]
\begin{center}
\small
\begin{tikzpicture}[scale=0.5]
\draw[<-,thick] (-4,0) -- (0,0);
\draw[->,thick] plot[domain=0:2] ({(\x^2)},{\x});
\draw[->,thick] plot[domain=0:2] ({(\x^2)},{-\x});
\draw (3,0) node[right]{(i)};
\draw (4,2) node[right]{(ii)};
\draw (4,-2) node[right]{(ii)};
\draw (0,3) node[left]{(iii)};
\draw (0,-3) node[left]{(iii)};
\draw (-4,0) node[left]{(iv)};
\draw[->,dotted] (0,0)--(2,0);
\draw[->,dotted] (0,0)--(0,2);
\draw (2,0) node[below]{$\alpha$};
\draw (0,2) node[right]{$\beta$};
\draw[solid] (0,0) circle (4pt);
\path [draw=none,fill=white, fill opacity = 1] (0,0) circle (4pt);
\draw[-] (-6,-4)--(6,-4)--(6,4)--(-6,4)--(-6,-4);
\end{tikzpicture}
\caption{Curves and regions (i)--(iv) in Theorem~\ref{T:main}}\label{F:domain}
\end{center}
\end{figure}
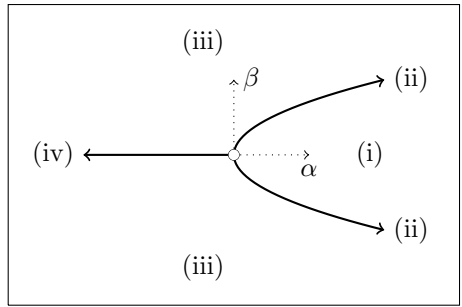

\bigskip

Figure~\ref{F:contour} is a contour map of the function $\mu\mapsto\|(V-\mu I)^{-1}\|$, computed using the theorem,
and produced using Maple.

\begin{figure}[ht]
\begin{center}
\includegraphics[scale=0.5, trim= 0 350 0 50]{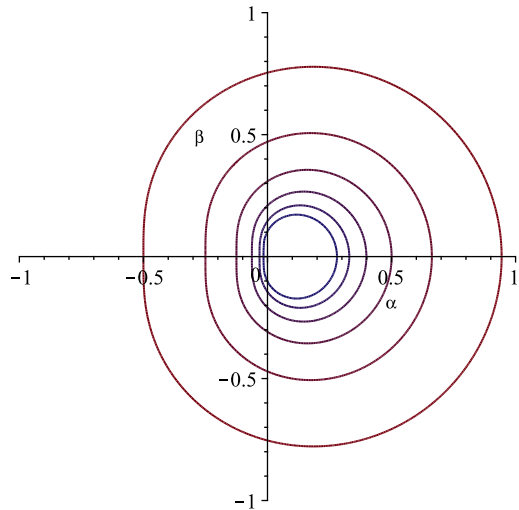}
\caption{Contour map of $\mu\mapsto\|(V-\mu I)^{-1}\|$. 
From the outside in, the curves are at levels $2,4,8,16,32,64$}
\label{F:contour}
\end{center}
\end{figure}

%%%%%%%%%%%%%%%%%%%%%%%%%%%%%%%%%%%%%

\section{Remarks on Theorem~\ref{T:main}}\label{S:remarks}

(1) A quick inspection of the statement of the theorem reveals a reflectional symmetry $\beta\leftrightarrow-\beta$.
In other words, we have 
\[
\|(V-\mu I)^{-1}\|=\|(V-\overline{\mu})^{-1}\|\quad(\mu\in\CC\setminus\{0\}).
\]
This  could also be deduced from the fact that $V=U^{-1}V^*U$, where $U:L^2[0,1]\to L^2[0,1]$ is the 
unitary operator $Uf(x):=f(1-x)$. Indeed, this implies that $(V-\mu I)^{-1}$ is unitarily equivalent to the adjoint of
$(V-\overline{\mu}I)^{-1}$, which of course has the same norm as $(V-\overline{\mu}I)^{-1}$.

\bigskip

(2) Part~\ref{I:iv} of the theorem can be reformulated as saying that
\[
\|(I-\nu V)^{-1}\|=1 \quad(\nu\in(-\infty,0]).
\]
This fact has been known for a long time. A short direct proof can be found in \cite[p.302]{Ha82}.
Note also that parts~\ref{I:i}--\ref{I:iii} together  imply the converse result, namely
\[
\|(I-\nu V)^{-1}\|>1 \quad(\nu\in\CC\setminus(-\infty,0]).
\]

\bigskip

(3) If $\alpha=0$, then there is an explicit expression for the norm of the resolvent of $V$.

\begin{theorem}\label{T:imag}
For $\beta\ne0$,
\[
\| (V-i\beta I)^{-1}\|=\frac{1+\sqrt{1+\beta^2\pi^2}}{\pi\beta^2}.
\]
\end{theorem}

\begin{proof}
By part~\ref{I:iii} of the Theorem~\ref{T:main}, applied with $\alpha=0$, we have $\|(V-i\beta I)^{-1}\|=1/\gamma$,
where $\gamma/(\beta^2-\gamma^2)\in(0,\pi)$ and $\cot(\gamma/(\beta^2-\gamma^2))=0$.
These conditions imply that $\gamma/(\beta^2-\gamma^2)=\pi/2$. Solving this quadratic equation for $\gamma$
leads directly to the result.
\end{proof}

Theorem~\ref{T:imag} can be reformulated equivalently as saying that
\[
\|(I-i\tau V)^{-1}\|=\frac{\tau}{\pi}+\sqrt{1+\frac{\tau^2}{\pi^2}} \quad(\tau\in\RR).
\]
It is interesting to note that, using the value of $\|I+i\tau V\|$ already calculated in \cite[Corollary~2.4]{LT10}, we have
\[
\|(I-i\tau V)^{-1}\|=\|I+i\tau V\| \quad(\tau\in\RR).
\]
This could also have been deduced by more abstract means,
since $(I-i\tau V)^{-1}=W(I+i\tau V)W^{-1}$,
where $W:L^2[0,1]\to L^2[0,1]$ is the unitary operator given by $W f(x):=e^{i\tau x}f(x)$.
This identity is an observation due to
Thomas Pedersen, see \cite[p.15]{Al97}.

\bigskip

(4) By part~\ref{I:iv} of Theorem~\ref{T:main}, we have $\|(V-\alpha I)^{-1}\|=1/|\alpha|$ for all $\alpha<0$.
The following result shows that, by contrast, $\|(V-\alpha I)^{-1}\|$ really explodes when $\alpha$ tends to zero from above.

\begin{theorem}\label{T:explode}
As $\alpha\to0^+$,
\[
\|(V-\alpha I)^{-1}\|=\frac{e^{1/\alpha}}{2\alpha}+O\Bigl(\frac{e^{-1/\alpha}}{\alpha^2}\Bigr).
\]
\end{theorem}

\begin{proof}
By part~\ref{I:i} of Theorem~\ref{T:main}, if $\alpha>0$, then
\[
\| (V-\alpha I)^{-1}\|=\frac{1}{\sqrt{\alpha^2-\gamma^2}},
\]
where $\gamma$ is the unique solution in $(0,\alpha)$ to $\coth(1/\gamma)=\alpha/\gamma$.
As $\alpha\to0^+$, we also have $\gamma\to0^+$, so
\[
\gamma=\alpha\tanh(1/\gamma)=\alpha(1+O(e^{-2/\gamma}))=\alpha(1+O(e^{-2/\alpha})).
\]
Hence, as $\alpha\to0^+$,
\begin{align*}
\| (V-\alpha I)^{-1}\|&=\frac{1}{\sqrt{\alpha^2-\gamma^2}}=\frac{1}{\sqrt{\alpha^2-\alpha^2\tanh(1/\gamma)}}=\frac{1}{\alpha}\cosh(1/\gamma)\\
&=\frac{1}{\alpha}\cosh\Bigl(\frac{1}{\alpha}(1+O(e^{-2/\alpha}))\Bigr)=\frac{e^{1/\alpha}}{2\alpha}+O\Bigl(\frac{e^{-1/\alpha}}{\alpha^2}\Bigr).\qedhere
\end{align*}
\end{proof}

%%%%%%%%%%%%%%%%%%%%%%%%%%%%%%%%%%%%%

\section{Proof of Theorem~\ref{T:main}}\label{S:proof}

The proof that follows is essentially a rehash of the ideas in \cite{LT10}.

Consider the operator $(V-\mu I)^*(V-\mu I)$, where $\mu=\alpha+i\beta\in\CC\setminus\{0\}$.
It equal to $|\mu|^2I$ plus a compact self-adjoint operator,
so its spectrum consists of $|\mu|^2$, together with a sequence of real eigenvalues converging to~$|\mu|^2$.
Further, since $(V-\mu V)^*(V-\mu I)$ is an invertible positive operator, 
all the eigenvalues are strictly positive. If the smallest eigenvalue $\lambda$ is less than $|\mu|^2$,
then $\|(V-\mu I)^{-1}\|=1/\sqrt{\lambda}$. Otherwise $\|(V-\mu I)^{-1}\|=1/|\mu|$.

Therefore we seek to identify all the eigenvalues $\lambda$ of $(V-\mu I)^*(V-\mu I)$
such that $\lambda<|\mu|^2$.
Let $\lambda$ be such an eigenvalue, say $\lambda=|\mu|^2-\delta$, where $\delta\in(0,|\mu|^2)$.
Then there exists $f\in L^2[0,1]\setminus\{0\}$ such that
\[
(V-\mu I)^*(V-\mu I)f=(|\mu|^2-\delta) f.
\]
Explicitly, this means that
\begin{equation}\label{E:explicit}
\delta f(x)-\overline{\mu}\int_0^x f-\mu\int_x^1f +\int_x^1\int_0^t f(u)\,du\,dt=0\quad\text{a.e.}
\end{equation}
This expresses $f(x)$ as the integral of an $L^2$ function,
so, after changing $f$ on a set of measure zero, we may suppose that $f$ is continuous. Using \eqref{E:explicit} to `bootstrap', we deduce that in fact $f\in C^\infty[0,1]$.

For $f\in\ C^\infty[0,1]$, the integral equation \eqref{E:explicit} is equivalent to the ODE 
\begin{equation}\label{E:ode}
\delta f''(x)+2i\beta f'(x)- f(x)=0,
\end{equation}
together with the two boundary conditions
\begin{equation}\label{E:bc}
\left\{
\begin{aligned}
\delta f'(0)+2i\beta f(0)&=0,\\
\delta f(1)-\overline{\mu}\int_0^1 f&=0.
\end{aligned}
\right.
\end{equation}
If we try the solution $f(x):=e^{\omega x}$ in \eqref{E:ode}, 
then we obtain the characteristic polynomial
\begin{equation}\label{E:quadratic}
\delta \omega^2+2i\beta\omega-1=0.
\end{equation}
The argument now divides into three cases, depending on whether the
discriminant of this polynomial, namely $4(\delta-\beta^2)$,
is positive,  zero or negative.
Note that the discriminant is equal to $4(\alpha^2-\lambda)$,
so the three cases correspond to whether $\lambda$ is less than, equal to, or greater than $\alpha^2$.
We shall treat them in that order, since we are seeking the smallest eigenvalue~$\lambda$.

\bigbreak

\textbf{Case I:} $\delta>\beta^2$.

In this case, the general solution to \eqref{E:ode} is
\[
f(x)=A_1e^{\omega_1x}+A_2e^{\omega_2 x},
\]
where $A_1,A_2\in\CC$, and $\omega_1,\omega_2$ are the solutions to the quadratic equation
\eqref{E:quadratic}, namely 
\[
\omega_1,\omega_2=\frac{-i\beta\pm\sqrt{\delta-\beta^2}}{\delta}.
\]
Substituting the expression for $f(x)$ into \eqref{E:bc}, we obtain
\begin{equation}\label{E:bcexplicit1}
\left\{
\begin{aligned}
\delta(A_1\omega_1+A_2\omega_2)+2i\beta (A_1+A_2)&=0,\\
\delta(A_1e^{\omega_1}+A_2e^{\omega_2})-\overline{\mu}\Bigl((A_1/\omega_1)(e^{\omega_1}-1)+(A_2/\omega_2)(e^{\omega_2}-1)\Bigr)&=0.
\end{aligned}
\right.
\end{equation}
Since  $2i\beta=-\delta(\omega_1+\omega_2)$,
the first condition in \eqref{E:bcexplicit1} simplifies to
\[
A_1/\omega_1+A_2/\omega_2=0.
\]
Substituting this information into the second condition in \eqref{E:bcexplicit1}, 
and simplifying, we obtain the relation
\[
\delta(\omega_1e^{\omega_1}-\omega_2e^{\omega_2})(A_1/\omega_1)
-\overline{\mu}(e^{\omega_1}-e^{\omega_2})(A_1/\omega_1)=0.
\]
Since $f\not\equiv0$, we must have $A_1\ne0$, and so this last relation becomes
\begin{equation}\label{E:messy}
\delta(\omega_1e^{\omega_1}-\omega_2e^{\omega_2})
-\overline{\mu}(e^{\omega_1}-e^{\omega_2})=0.
\end{equation}
Now, recalling that $\mu=\alpha+i\beta$ and $\omega_j=(-i\beta\pm \gamma)/\delta$, 
where
\[
\gamma:=\sqrt{\delta-\beta^2},
\]
we can rearrange \eqref{E:messy} to get
\[
e^{2\gamma/\delta}=\frac{\alpha+\gamma}{\alpha-\gamma},
\]
which in turn is equivalent to
\begin{equation}\label{E:gammaeqn1}
\coth\Bigl(\frac{\gamma}{\beta^2+\gamma^2}\Bigr)=\frac{\alpha}{\gamma}.
\end{equation}
So we seek solutions $\gamma$ to \eqref{E:gammaeqn1} such that $\gamma>0$.
We have to consider various sub-cases. 

\begin{itemize}
\item If $\alpha\le 0$, then obviously no solution to \eqref{E:gammaeqn1} with $\gamma>0$ is possible.
\item If $\alpha>0$ and $\beta=0$, then the left-hand side of \eqref{E:gammaeqn1} reduces to $\coth(1/\gamma)$. 
Now, as $\gamma$ increases from $0$ to $\infty$,
the function  $\coth(1/\gamma)$ increases from $1$ to $+\infty$,
whereas $\alpha/\gamma$ decreases from $+\infty$ to $0$.
Hence, in this case, there is a unique solution $\gamma>0$ to  \eqref{E:gammaeqn1}.
Clearly $\gamma\in(0,\alpha)$.
\item If $\alpha>0$ and $\beta\ne0$, then solving \eqref{E:gammaeqn1} amounts to finding the positive fixed points of the 
function $\Phi(\gamma):=\alpha\tanh(\gamma/(\beta^2+\gamma^2))$.
This function is increasing and concave for $0\le \gamma\le|\beta|$ and decreasing for $\gamma\ge|\beta|$. 
Therefore $\Phi$ has a positive fixed point if and only if $\Phi'(0)>1$,
and in this case the solution is unique and lies in $(0,\alpha)$. A calculation shows that $\Phi'(0)=\alpha/\beta^2$, so the condition for a solution to exist is that $\alpha>\beta^2$. 
\end{itemize}

To summarize, we have proved that,
if $\alpha>\beta^2$, then \eqref{E:gammaeqn1} has a unique solution $\gamma\in(0,\alpha)$. Working backwards, we see that the corresponding value of $\lambda$, namely
\[
\lambda=|\mu|^2-\delta=|\mu|^2-(\beta^2+\gamma^2)=\alpha^2-\gamma^2,
\]
is the smallest eigenvalue of $(V-\mu I)^*(V-\mu I)$. This establishes part~\ref{I:i} of the theorem.

We have also proved that, if $\alpha\le \beta^2$, then \eqref{E:gammaeqn1} has no positive solution.
So, in this case,  there are no eigenvalues $\lambda=|\mu|^2-\delta$ of $(V-\mu I)^*(V-\mu I)$ such that
$\delta>\beta^2$, i.e.,  no eigenvalues satisfy $\lambda<\alpha^2$.

\bigbreak

\textbf{Case II:} $\delta=\beta^2$.

In this case, the general solution to \eqref{E:ode} is
\[
f(x)=Ae^{\omega x}+Bxe^{\omega x},
\]
where $A,B\in\CC$, and $\omega$ is the (double) root of the quadratic equation
\eqref{E:quadratic}, namely 
\[
\omega=-i\beta/\delta.
\]
Substituting the expression for $f$ into \eqref{E:bc}, we obtain
\begin{equation}\label{E:bcexplicit2}
\left\{
\begin{aligned}
\delta(A\omega+B)-2i\beta A&=0,\\
\delta(Ae^{\omega}+Be^{\omega})-\overline{\mu}\Bigl((A/\omega)(e^{\omega}-1)+(B/\omega)e^\omega-(B/\omega^2)(e^{\omega}-1)\Bigr)&=0.
\end{aligned}
\right.
\end{equation}
Since  $i\beta=-\delta\omega$,
the first condition in \eqref{E:bcexplicit2} simplifies to
\[
B=\omega A.
\]
Substituting for $B$ in the second condition in \eqref{E:bcexplicit2}, 
and simplifying, we obtain the relation
\[
\delta(1+\omega)A-\overline{\mu} A=0.
\]
Since $f\not\equiv0$, we must have $A\ne0$, and so this last relation becomes
\[
\delta(1+\omega)-\overline{\mu}=0.
\]
Since $\delta\omega=-i\beta$ and $\mu=\alpha+i\beta$, the last relation reduces simply to
\[
\delta=\alpha.
\]
In conjunction with the condition $\delta=\beta^2$, this implies that 
$\alpha=\beta^2$ and $\lambda=\alpha^2$.

To summarize, we have proved that, if $\alpha=\beta^2$,
then the smallest eigenvalue of  $(V-\mu I)^*(V-\mu I)$ is equal to $\alpha^2$.
This establishes part~\ref{I:ii} of the theorem.

\bigbreak

\textbf{Case III:} $\delta<\beta^2$.

In this case, the general solution to \eqref{E:ode} is
\[
f(x)=A_1e^{\omega_1x}+A_2e^{\omega_2 x},
\]
where $A_1,A_2\in\CC$, and $\omega_1,\omega_2$ are the solutions to the quadratic equation
\eqref{E:quadratic}, namely 
\[
\omega_1,\omega_2=\frac{-i\beta\pm i\sqrt{\beta^2-\delta}}{\delta}.
\]
Repeating the analysis carried out in Case~I leads to the equation
\[
e^{2i\gamma/\delta}=\frac{\alpha+i\gamma}{\alpha-i\gamma},
\]
where now
\[
\gamma:=\sqrt{\beta^2-\delta}\in(0,\,|\beta|).
\]
This is equivalent to saying that
\[
\frac{2\gamma}{\delta}
=2\arg\Bigl(\alpha+i\gamma\Bigr)+2n\pi \quad(n\in\ZZ),
\]
which in turn is equivalent to the relation
\begin{equation}\label{E:gammaeqn3}
\cot\Bigl(\frac{\gamma}{\beta^2-\gamma^2}\Bigr)
=\frac{\alpha}{\gamma}.
\end{equation}
So we seek solutions $\gamma$ to \eqref{E:gammaeqn3} such that $0<\gamma<|\beta|$,
in particular the smallest such solution.

Just as in Case~I, we have to consider various sub-cases. 
\begin{itemize}
\item If $\beta=0$, then obviously no solution $\gamma$ with $0<\gamma<|\beta|$ is possible.
\item If $\alpha=0$, then \eqref{E:gammaeqn3} holds if and only if there is an integer $n$ 
such that $\gamma/(\beta^2-\gamma^2)=\pi/2+n\pi$.
The smallest positive $\gamma$ corresponds to taking $n=0$, i.e.
$\gamma/(\beta^2-\gamma^2)=\pi/2$.
\item If $\alpha\ne0$ and $\beta\ne0$, then solving \eqref{E:gammaeqn3} is equivalent to finding a positive fixed
point $\gamma$ of the function
\[
\Psi(\gamma):=\alpha\tan\Bigl(\frac{\gamma}{\beta^2-\gamma^2}\Bigr).
\]
Let $\gamma_1,\gamma_2$ be the positive  solutions $\gamma$ of the equation 
$\gamma/(\beta^2-\gamma^2)=\pi/2,\pi$ respectively.
\begin{itemize}
\item If $\alpha<0$, then $\Psi(\gamma)<0$ for $\gamma\in(0,\gamma_1)$,
so there is no fixed point for $\gamma$ in this range. 
Also, as $\gamma$ runs from $\gamma_1$ up to $\gamma_2$, 
the function $\Psi(\gamma)$ decreases from $+\infty$ to $0$, 
so there is a unique fixed point for $\gamma$ in this range. 
It is clearly the smallest positive fixed point.
\item If $\alpha>0$, then $\Psi(0)=0$ and $\Psi(\gamma_1^-)=+\infty$, 
and $\Psi(\gamma)$ is convex and increasing on $(0,\gamma_1)$.
Thus $\Psi(\gamma)$ has a fixed point in $(0,\gamma_1)$ if and only if $\Psi'(0)<1$, 
and in this case the fixed point is unique, and it is clearly the smallest positive fixed point.
A computation gives $\Psi'(0)=\alpha/\beta^2$, so the condition $\Psi'(0)<1$ translates to $\alpha<\beta^2$.
\end{itemize}
\end{itemize}

To summarize, we have proved that, if $\alpha<\beta^2$ and $\beta\ne0$, 
then \eqref{E:gammaeqn3} has a unique positive solution $\gamma$ such that  $\gamma/(\beta^2-\gamma^2)\in(0,\pi)$,
and this is the smallest positive solution. 
Working backwards, we see that the corresponding
value of $\lambda$, namely
\[
\lambda=|\mu|^2-\delta=|\mu|^2-(\beta^2-\gamma^2)=\alpha^2+\gamma^2,
\]
is the smallest eigenvalue of $(V-\mu I)^*(V-\mu I)$. This establishes part~\ref{I:iii}  of the theorem.

We have also proved that, if $\alpha<0$ and $\beta=0$,
then there are no eigenvalues $\lambda$ of $(V-\mu I)^*(V-\mu I)$ 
such that $\lambda<|\mu|^2$. Consequently
$\|(V-\mu I)^{-1}\|=1/|\mu|=1/|\alpha|$ in this case.
This establishes part~\ref{I:iv} of the theorem, and concludes the proof.\qed

\end{document}